\documentclass{article}

\usepackage{graphicx} 
\usepackage{amsthm}
\usepackage{amsmath}
\usepackage{amssymb}
\usepackage{mathrsfs}
\usepackage{tikz}
\usepackage{mathtools}
\usepackage{enumitem}
\usepackage{comment}
\usepackage[hidelinks]{hyperref}

\newcommand{\bb}{\mathbb}
\newcommand{\mc}{\mathcal}

\DeclareMathOperator{\conv}{conv}

\DeclareMathOperator{\ini}{in}
\DeclareMathOperator{\rint}{rint}

\numberwithin{equation}{section}

\theoremstyle{definition}

\newtheorem{thm}[equation]{Theorem}
\newtheorem{prop}[equation]{Proposition}
\newtheorem{lem}[equation]{Lemma}
\newtheorem{cor}[equation]{Corollary}
\newtheorem{conj}[equation]{Conjecture}

\newtheorem{rem}[equation]{Remark}
\newtheorem{open}[equation]{Open Problem}

\newtheorem{ques}[equation]{Question}

\setlist{nosep}

\title{
The Gr\"obner Version of 
\\ White's Conjecture is False}
\author{
Spencer Backman\footnote{
University of Vermont,
Department of Mathematics and Statistics,
Innovation Hall, E220 82 University Place,
Burlington, VT, USA, 05405,
\texttt{sbackman@uvm.edu}}
\and 
Nathan Cheung\footnote{
University of Washington,
Department of Mathematics,
Box 354350, C-138 Padelford,
Seattle, WA 98195-4350,
\texttt{ncheung@uw.edu}}
\and 
Micha{\l} Laso\'n\footnote{
Institute~of~Mathematics~of~the~Polish~Academy~of~Sciences,
ul.~\'Sniadeckich~8, 00-656~Warszawa, Poland,
\texttt{michalason@gmail.com}}
\and 
Gaku Liu\footnote{
University of Washington,
Department of Mathematics,
Box 354350, C-138 Padelford,
Seattle, WA 98195-4350,
\texttt{gakuliu@uw.edu}
}
\and 
Mateusz Micha{\l}ek
\footnote{University of Konstanz, Germany,
Fachbereich Mathematik und Statistik, Fach D 197,
D-78457 Konstanz, Germany,
\texttt{mateusz.michalek@uni-konstanz.de}}
}
\date{June 12, 2026}

\begin{document}
\maketitle
\begin{abstract}
We show that the toric ideal of the Fano matroid polytope does not have a quadratic Gr\"obner basis. This resolves in the negative a strong version of White's conjecture from matroid theory. This result was found independently by De Loera, Ferroni, Morales, and Rambau \cite{DeLoeraFerroniMoralesRambau2026}. Our approach is based on a new characterization of regular unimodular flag triangulations, which reduces the problem to an instance of SMT involving boolean and real variables. We then use an SMT solver to prove unsatisfiability. Using this approach, we also show that all 8-element matroids which do not have the Fano matroid or its dual as a minor, with the possible exception of the matroid $T_8$, have toric ideals which admit quadratic Gr\"obner bases.
\end{abstract}
\section{Introduction}

In this article we investigate algebraic and geometric aspects of the toric ideal $I_M$ of a matroid base polytope $P(M)$ for a given matroid $M$.  This is a classical topic at the intersection of combinatorics, algebraic geometry, commutative algebra, and polyhedral geometry.  There are many important results in area; here we briefly recall two:

\begin{itemize}
    \item  The closure of the torus orbit of any point $L$ in a Grassmannian is defined by the ideal $I_M$ for the matroid associated to $L$ \cite{gelfand1987combinatorial}.
    \item The matroid base polytope $P(M)$ is always normal \cite{white1977basis}.
\end{itemize}

Still, many properties of $P(M)$ and $I_M$ remain only conjectural. The famous White's conjecture predicts quadratic generation of $I_M$.\footnote{White presented three versions of the conjecture. Two of those are known to be equivalent \cite{lason2014toric}.}
\begin{conj}[White's conjecture]
    The ideal $I_M$ is generated in degree two.
\end{conj}
The conjecture is only known for specific classes of matroids \cite{Blasiak2008GraphicMatroidToricIdeal, Kashiwabara2010Rank3MatroidToricIdeal, Schweig2011LatticePathMatroids, Bonin2013SparsePavingWhite, lason2014toric, Shibata2016SeriesParallelMatroids, McGuinness2020FrameMatroidsWhite, BercziSchwarcz2024SplitMatroids, YuYuen2025PavingMatroidsWhite, HanMichalekWeigert2025WhiteInnerProjections}. It is also known that the quadrics define the same projective scheme as $I_M$ \cite{lason2014toric}.
Herzog and Hibi proposed the following strengthening of White's conjecture \cite{herzog2002discrete}.
\begin{ques}
    Does the ideal $I_M$ always have a quadratic Gr\"obner basis?
\end{ques}
The answer is positive in small cases, see  \cite{HayaseHibiKatsunoShibata2022}, as well as for the class of base-sortable matroids \cite{blum2001base}, which includes uniform matroids \cite{sturmfels1996grobner} and lattice path matroids \cite{Schweig2011LatticePathMatroids}. In particular, \cite{HayaseHibiKatsunoShibata2022} identifies the Fano matroid as the smallest matroid where the answer to the question was not known.  
The main result of this paper shows that when $M$ is the Fano matroid, then $I_M$ has no quadratic Gr\"obner basis. By a general result of Sturmfels \cite{sturmfels1996grobner}, the existence of a quadratic Gr\"obner basis for $I_M$ is equivalent to the existence of a regular, unimodular, flag triangulation, also known as a quadratic triangulation, of the polytope $P(M)$. 

Our approach is to prove that no quadratic triangulation exists for the Fano matroid. 
For this, we present in Theorem \ref{thm:main conditions} new necessary and sufficient conditions for the existence of quadratic triangulations. These conditions can be expressed as formal statements in Boolean and real variables. 
Using the SMT solver cvc5 \cite{cvc5} we obtain a formal proof that the conditions cannot be simultaneously satisfied.  Previous work of authors Backman and Liu proved that every matroid base polytope admits a regular, unimodular, but not necessarily flag, triangulation \cite{BackmanLiu2025MatroidBasePolytope}.  In light of the current work, the result of Backman and Liu is, in a sense, best possible.

Using our methods, we also show that the polytope of every matroid with 8 elements and no Fano or Fano dual minor has a quadratic triangulation, with the possible exception of one matroid called $T_8$. See Section 5 and Open Problem~\ref{T8}.

While preparing our preprint, the independent work \cite{DeLoeraFerroniMoralesRambau2026} appeared on arXiv, proving the same result. Their approach is to reduce the problem to an instance of Boolean satisfiability and use a SAT solver. In our approach, we reduce to a mixed satisfiability problem with Boolean and real variables, and use an SMT solver. It would be interesting to see if ideas from both can be used to produce more efficient algorithms; for example, to resolve whether $I_{T_8}$ has a quadratic Gr\"obner basis.

\section*{Acknowledgements}
 SB was supported by NSF Grant (DMS-2246967) and Simons Foundation Gift \# 854037. GL was supported by NSF grant DMS-2348785. ML was supported by the National Science Centre, Poland, grant no.~2019/34/E/ST1/00087. MM was supported by the DFG grant 580118961.

We thank Michael Joswig, Matt Larson, Igor Makhlin, Sam Payne, and Benjamin Schr\"oter  for helpful comments and conversations.  We thank J\"org Rambau for sharing an early version of TOPCOM 1.2.0c.

\section*{Tool and computational resource disclosure}

Programs cvc5, Macaulay2, Sage, Singular, and TOPCOM 1.2.0c.  were utilized in this project.  
ChatGPT was consulted during writing the Python script that generates the SMT-LIB file, although the final code was written by hand.  OpenAI's Codex and Anthropic's Claude were used for scripting the 8-element matroid pipeline described in Subsection \ref{subsec: 8 elem}.  We emphasize that these coding agents were primarily used for code generation and campaign implementation, but not for pipeline design with a single exception: a helpful, albeit unsolicited, suggestion was made by a coding agent which improved the QTV step of the pipeline.  Apart from this contribution, the mathematics in this article was entirely human generated and did not utilize AI.  The 8-element matroid computations were performed on the Vermont Advanced Computing Cluster.  

\section{Preliminaries}

\subsection{Affine configurations}

Let $\mc A$ be a subset of $\bb R^d$. Let $\bb R^{\mc A}$ be the set of all real-valued tuples $(\lambda_a)_{a \in \mc A}$ indexed by $\mc A$. An \emph{affine dependency} is any $(\lambda_a) \in \bb R^{\mc A}$ such that $\sum_{a \in \mc A} \lambda_a a = 0$ and $\sum_{a \in \mc A} \lambda_a = 0$. 

A \emph{simplex} of $\mc A$ is an affinely independent subset of $\mc A$. Given a simplex $\sigma$, we refer to the convex hull of $\sigma$ in $\bb R^d$, denoted $\conv \sigma$, as a \emph{geometric simplex}. Any point $x$ in $\conv \sigma$ can be written uniquely as $x = \sum_{a \in \sigma} \lambda_a a$ where $0 \le \lambda_a \le 1$ and $\sum_{a \in \sigma} \lambda_a = 1$. We call this the \emph{convex hull representation} of $x$ with respect to $\sigma$. The \emph{relative interior} of $\sigma$, denoted $\rint \sigma$, is the set of all $x \in \conv \sigma$ whose convex hull representation $x = \sum_{a \in \sigma} \lambda_a a$ satisfies $\lambda_a > 0$ for all $a$.

Let $h : \mc A \to \bb R$ be any function. For any simplex $\sigma$ of $\mc A$ and point $x \in \conv \sigma$, we define
\[
h(x,\sigma) := \sum_{a \in \sigma} \lambda_a h(a)
\]
where $x = \sum_{a \in \sigma} \lambda_a a$ is the convex hull representation of $x$ with respect to $\sigma$.

The \emph{support} of an affine dependency $(\lambda_a)$ is the set of all $a \in \mc A$ such that $\lambda_a \neq 0$. The \emph{signed support} of $(\lambda_a)$ is the ordered pair $(S_+,S_-)$ where $S_+$ (respectively, $S_-$) is the set of all $a \in \mc A$ such that $\lambda_a > 0$ (respectively, $\lambda_a < 0$).

An \emph{(affine) circuit} of $\mc A$ is a minimal affinely dependent subset of $\mc A$. In other words, a circuit is a subset of $\mc A$ which is affinely dependent, and any proper subset is affinely independent. For every circuit $C$, there is an affine dependency $(\lambda_a)$ whose support is $C$, and this dependency is unique up to multiplication by a nonzero number. We call the signed support $(C_+,C_-)$ of this $(\lambda_a)$ a \emph{signed circuit} (associated to $C$); every circuit has two associated signed circuits $(C_+,C_-)$ and $(C_-,C_+)$. The pair of numbers $(|C_+|,|C_-|)$ is called the \emph{signature} of the signed circuit. The set partition $\{C_+,C_-\}$ is the unique bipartition of $C$ such that the relative interiors of $C_+$ and $C_-$ intersect. By the minimality of circuits, we obtain the following proposition. 

\begin{prop}
If $(S_+,S_-)$ is the signed support of a nonzero affine dependency, then there is a signed circuit $(C_+,C_-)$ such that $C_+ \subset S_+$ and $C_- \subset S_-$.
\end{prop}


\begin{cor} \label{rint}
Let $S$, $T \subset \mc A$. Then the relative interiors of $S$ and $T$ intersect if and only if there is a signed circuit $(C_+,C_-)$ with $C_+ \subset S$ and $C_- \subset T$.
\end{cor}

\subsection{Subdivisions of affine configurations}

Let $\mc A$ be a finite subset of $\bb R^d$. A \emph{subdivision} of $\mc A$ is a set $\mc S$ of subsets of $\mc A$ such that
\begin{enumerate}
    \item If $S \in \mc S$ and $F$ is a face of $\conv S$, then $F \cap S \in \mc S$.
    \item If $S$ and $T$ are different elements of $\mc S$, then the relative interiors of $S$ and $T$ are disjoint.
    \item $\bigcup_{S \in \mc S} \conv S = \conv \mc A$.
\end{enumerate}
A triangulation is a subdivision all of whose elements are simplices.

Let $h : \mc A \to \bb R$ be a function. Define
\[
\mc A^h := \{ (a,h(a)) : a \in \mc A \} \subset \bb R^{d+1}.
\]
Let $P = \conv \mc A^h$. A \emph{lower face} of $P$ is a face which has an outward normal vector with negative last coordinate. Let $p : \bb R^{d+1} \to \bb R^d$ be the projection which deletes the last coordinate. Then the collection of sets
\[
\{ p(F \cap \mc A^h) : F \text{ is a lower face of } P \}
\]
is a subdivision of $\mc A$, which we call the subdivision \emph{induced} by $h$. A subdivision which is induced by a function is called \emph{regular}.

\begin{prop} \label{induced}
Let $h : \mc A \to \bb R$ be any function and $\mc S$ the subdivision induced by $h$. Let $\sigma$ be a simplex of $\mc A$. The following are equivalent.
\begin{enumerate}[label=(\alph*)]
    \item $\sigma \in \mc S$.
    \item For any $x \in \rint \sigma$ and all simplices $\tau \subset \mc A$ with $x \in \rint \tau$ and $\sigma \neq \tau$, we have $h(x,\sigma) < h(x,\tau)$.
    \item For some $x \in \rint \sigma$ and all simplices $\tau \subset \mc A$ with $x \in \rint \tau$ and $\sigma \neq \tau$, we have $h(x,\sigma) < h(x,\tau)$.
\end{enumerate}
\end{prop}

\begin{proof}
(a) $\implies$ (b). Assume $\sigma \in \mc S$. Then there is a linear functional $(\phi,-1) \in (\bb R^d)^\ast \times \bb R$ and real number $b$ such that $(\phi,-1)(a,h(a)) = b$ for all $a \in \sigma$ and $(\phi,-1)(a,h(a)) < b$ for all $a \in \mc A \setminus \sigma$. Let $x \in \rint \sigma$, and write $x = \sum_{a \in \sigma} \lambda_a a$ where $\lambda_a > 0$ for all $a \in \sigma$ and $\sum_{a \in \sigma} \lambda_a = 1$. We have
\[
\sum_{a \in \sigma} \lambda_a (\phi,-1)(a,h(a)) = b \implies \phi(x) - h(x,\sigma) = b \implies h(x,\sigma) = \phi(x) - b.
\]
Consider any simplex $\tau \subset \mc A$ with $x \in \rint \tau$ and $\sigma \neq \tau$. In particular, at least one element of $\tau$ is not in $\sigma$. Write $x = \sum_{a \in \tau} \lambda'_a a$ where $\lambda'_a > 0$ for all $a \in \tau$ and $\sum_{a \in \tau} \lambda'_a = 1$. We have
\[
\sum_{a \in \tau} \lambda'_a (\phi,-1)(a,h(a)) < b \implies \phi(x) - h(x,\tau) < b \implies h(x,\tau) > \phi(x) - b.
\]

(b) $\implies$ (c) is straightforward.

(c) $\implies$ (a). There exists a lower face $F$ of $\conv(\mc A^h)$ such that $x\in \rint p(F\cap \mc A^h)$. Fix a linear functional $(\phi,-1) \in (\bb R^d)^\ast \times \bb R$ and real number $b$ such that $(\phi,-1)(a,h(a)) = b$ for all $a \in p(F\cap \mc A^h)$ and $(\phi,-1)(a,h(a)) < b$ for all $a \in \mc A \setminus p(F\cap \mc A^h)$. Take $\tau$ to be any simplex contained in $p(F\cap \mc A^h)$ such that $x\in \rint \tau$. By the same argument as in the first implication, we see that $h(x,\sigma)\geq h(x,\tau)$ and thus, by assumption (c) we have $\sigma=\tau$. If $p(F\cap \mc A^h)$ is not a simplex, then there would exist two distinct $\tau$ in $p(F\cap \mc A^h)$ with $x\in \rint \tau$, which is not possible since they all must equal $\sigma$. Thus $p(F\cap \mc A^h)$ equals $\sigma$ and $\sigma\in \mc S$.
\end{proof}

Let $\mc T$ be a triangulation of $\mc A$. A \emph{non-face} of $\mc T$ is a subset of $\mc A$ which is not in $\mc T$. We say $\mc T$ is \emph{flag} if all its minimal non-faces have cardinality at most 2. In other words, $\mc T$ is flag if for every subset $S \subset \mc A$ which is not in $\mc T$ and $|S| > 2$, there is some $S' \subset S$ which is not in $\mc T$ with $|S'| = 2$.

Now assume $\mc A \subset \bb Z^d$. A simplex $\sigma$ is \emph{unimodular} if there is an affine lattice isomorphism $\phi : \bb Z^d \to \bb Z^d$ such that $\phi(\sigma) = \{0,e_1,\dots,e_k\}$ for some $k \ge 0$, where $e_i$ is a standard basis vector. A triangulation is \emph{unimodular} if all of its elements are unimodular.

A \emph{quadratic triangulation} is a triangulation which is regular, flag, and unimodular.

\subsection{Gr\"{o}bner bases and toric ideals}\label{subsec: gb and tor ideal}

Fix a field $\bb K$, and let $R = \bb K[t_1,\dots,t_n]$ be a polynomial ring. A \emph{monomial order} is a total order $\prec$ on the set of monomials $M$ of $R$ such that $u \prec v$ implies $uw \prec vw$ for all $w \in M$, and $1 \prec u$ for all $u \in M \setminus \{1\}$. Fix some monomial order $\prec$. For any polynomial $f \in R$, the \emph{leading term} of $f$ is the largest monomial appearing in $f$ with respect to $\prec$. For any ideal $I$ of $R$, the \emph{initial ideal} $\ini_\prec I$ is the ideal generated by all leading terms of elements of $I$. A \emph{Gr\"obner basis} for $I$ is a finite set $G \subset I$ such that the leading terms of elements of $G$ generate $\ini_\prec I$.

Let $h : [n] \to \bb R_{>0}$ be a function such that $h(1)$, \dots, $h(n)$ are linearly independent over $\bb Q$. For a monomial $u = t_1^{c_1} \dots t_n^{c_n} \in M$, define $h(u) = \sum_{i=1}^n h(i) c_i$. Define the relation $\prec_h$ in $M$ by $u \prec_h v$ if $h(u) < h(v)$. Then $\prec_h$ is a monomial order on $M$.

\begin{prop}[{\cite[Proposition~1.11]{sturmfels1996grobner}}]
For every monomial order $\prec$ and ideal $I$, there exists $h : [n] \to \bb R_{>0}$ such that $\ini_{\prec} I = \ini_{\prec_h} I$.
\end{prop}

Let $\mc A$ be a finite subset of $\bb Z^d$. Let $R_{\mc A} := \bb K[ t_a : a \in \mc A]$ be the ring of polynomials generated by variables $t_a$ indexed by $\mc A$. 
The \emph{toric ideal} $I_{\mc A}$ is the ideal of $R_{\mc A}$ generated by all binomials of the form
\[
\prod_{a \in S_+} t_a^{\lambda_a} - \prod_{a \in S_-} t_a^{-\lambda_a} 
\]
where $(\lambda_a)_{a \in \mc A}$ ranges over all affine dependencies of $\mc A$ with $\lambda_a \in \bb Z$ for all $a$, and $(S_+,S_-)$ is the signed support of $(\lambda_a)_{a \in \mc A}$.

Let $h : \mc A \to \bb R_{>0}$ be any function. If $h$ is generic, then it induces a monomial order $\prec_h$ on $R_{\mc A}$ as above. In addition, $h$ induces a triangulation $\mc T$ of $\mc A$. Let $I_{\mc T}$ be the ideal of $R_{\mc A}$ generated by all monomials of the form $\prod_{a \in N} t_a$, where $N$ is a non-face of $\mc T$. This is also known as the \emph{Stanley-Reisner ideal} of $\mc T$.

\begin{thm}[{\cite[Theorem~8.3 and Corollary~8.9]{sturmfels1996grobner}}] \label{groebner}
We have
\[
\sqrt{ \ini_{\prec_h} I_{\mc A}} = I_{\mc T}.
\]
In addition, $\ini_{\prec_h} I_{\mc A}$ is generated by squarefree monomials if and only if $\mc T$ is unimodular.
\end{thm}

\begin{cor}
$I_{\mc A}$ has a quadratic Gr\"obner basis with squarefree initial ideal if and only if $\mc A$ has a quadratic triangulation. If $\mc A$ is in convex position, then $I_{\mc A}$ has a quadratic Gr\"obner basis if and only if $\mc A$ has a quadratic triangulation.
\end{cor}

\subsection{Matroid polytopes}\label{subsec:mat polytopes}

Let $M = (E,\mc B)$ be a matroid with ground set $E$ and set of bases $\mc B$. For each $B \in \mc B$, let $e_B$ be the indicator vector of $B$ in $\bb R^E$, that is $e_B := \sum_{i \in B} e_i$ where $e_i$ is a standard basis vector of $\bb R^E$. Let $P(M) = \{e_B : B \in \mc B\}$. The polytope with vertex set $P(M)$ is the \emph{matroid (base) polytope} of $M$. (Since every element of $P(M)$ has 0-1 coordinates, $P(M)$ is in convex position.) By abuse of notation, when it does not lead to confusion, we will also denote by $P(M)$ the matroid base polytope. 

\begin{conj} \label{mainconjecture}
For every matroid $M$, the toric ideal $I_{P(M)}$ has a quadratic Gr\"obner basis. In other words, $P(M)$ has a quadratic triangulation.
\end{conj}

Let $F_7$ be the matroid represented by the set of nonzero elements of $\bb F_2^3$. This is known as the \emph{Fano matroid}. Its matroid polytope is 6-dimensional with 28 vertices. Our main result is the following.

\begin{thm}\label{thm: main fano}
$P(F_7)$ does not have a quadratic triangulation. Thus, Conjecture~\ref{mainconjecture} is false.
\end{thm}

\section{Conditions}

Let $\mc A$ be a subset of $\{0,1\}^d$. Let $\mc E = \binom{\mc A}{2}$, that is, the set of all 2-element subsets of $\mc A$. We define an equivalence relation $\sim$ on $\mc E$ by
\[
\{a,b\} \sim \{a',b'\} \qquad \text{if} \qquad a + b = a' + b'.
\]
A \emph{transversal} of $\sim$ is a set $\mc E' \subset \mc E$ containing exactly one representative of each equivalence class of $\sim$. The key result is the following.

\begin{thm}\label{thm:main conditions}
Let $\mc A$, $\mc E$, and $\sim$ be as above. Then $\mc A$ has a quadratic triangulation if and only if there exists a transversal $\mc E'$ of $\sim$ and a function $h : \mc A \to \bb R$ such that
\begin{enumerate}[label=(\arabic*)]
    \item \label{cond1} For every $\{a',b'\} \in \mc E'$ and every $\{a,b\} \in \mc E$ with $\{a,b\} \sim \{a',b'\}$ and $\{a,b\} \neq \{a',b'\}$, we have
    \[
    h(a') + h(b') < h(a) + h(b).
    \]
    \item \label{cond2} For every signed circuit $(C_+,C_-)$ of $\mc A$, we have
    \[
    \binom{C_+}{2} \not\subset \mc E' \quad \text{or} \quad \binom{C_-}{2} \not\subset \mc E'.
    \]
\end{enumerate}
\end{thm}

\begin{proof}
Suppose $\mc A$ has a quadratic triangulation $\mc T$. Let $\mc E'$ be the set of edges of $\mc T$. We claim $\mc E'$ is a transversal of $\sim$. Note that two elements of $\mc E$ are equivalent if and only if they have the same midpoint. Thus, since $\mc T$ is a triangulation, $\mc T$ contains at most one representative from each equivalence class of $\sim$. To see that $\mc T$ contains at least one representative, let $\mc C$ be an equivalence class of $\sim$, and let $x = (a+b)/2$ for any $\{a,b\} \in \mc C$. Since $x \in \conv \mc A$, there is some $\sigma \in \mc T$ such that $x \in \rint \sigma$. Since $\mc T$ is unimodular, after an affine unimodular transformation we may assume $\sigma = \{0,e_1,\dots,e_k\}$ for some $k \ge 0$, so $x = \sum_{i=1}^k \lambda_i e_i$ where $0 < \lambda_i < 1$ and $\sum_{i=1}^k \lambda_i < 1$. Since $x$ was originally the midpoint of different elements of $\{0,1\}^d$, after transformation all its coordinates are in $\frac{1}{2} \bb Z$, and at least one of its coordinates is not an integer. It follows that $k = 1$ and $\lambda_1 = 1/2$. Thus $\sigma$ is an edge and $x$ is the midpoint, so $\sigma \in \mc C$, as desired.

Now, let $h : \mc A \to \bb R$ be a function which induces $\mc T$. We claim $\mc E'$ and $h$ satisfy conditions \ref{cond1} and \ref{cond2}. Condition \ref{cond1} holds by applying Proposition~\ref{induced}(b) to the midpoint of any element of $\mc E'$. For condition \ref{cond2}, let $(C_+,C_-)$ be a circuit of $\mc A$. Since the relative interiors of $C_+$ and $C_-$ intersect, we must have either $C_+ \not\subset \mc T$ or $C_- \not\subset \mc T$. Since $\mc T$ is flag, it follows that either $\binom{C_+}{2} \not\subset \mc E'$ or $\binom{C_-}{2} \not\subset \mc E'$, as desired.

Conversely, suppose $\mc E'$ is a transversal of $\sim$ and $h : \mc A \to \bb R$ a function satisfying conditions \ref{cond1} and \ref{cond2}. Perturbing $h$ by a small amount and adding a constant to $h$ does not affect these conditions. Thus, we may assume $h$ induces a triangulation $\mc T$, as well as a monomial order $\prec_h$ on $R_{\mc A}$. Let $\mc T'$ be the clique complex of $\mc E'$; that is, the set of all sets $S \subset \mc A$ such that $\binom{S}{2} \subset \mc E'$. By Corollary~\ref{rint} and condition \ref{cond2}, the relative interiors of the elements of $\mc T'$ are pairwise disjoint. Now, by condition \ref{cond1}, every edge of $\mc T$ must be in $\mc E'$. Therefore, $\mc T \subset \mc T'$. But $\mc T$ is a triangulation, so the relative interiors of its elements cover all of $\conv \mc A$. Thus $\mc T' \setminus \mc T$ must be empty, so $\mc T = \mc T'$.

Thus, by Theorem~\ref{groebner}, we have
\[
\ini_{\prec_h} I_{\mc A} \subset \sqrt{\ini_{\prec_h} I_{\mc A}} = I_{\mc T'} = \langle t_a t_b : \{a,b\} \in \mc E \setminus \mc E' \rangle.
\]
In the last equality, we have used the fact that $\mc A$ is in convex position, so all non-faces of $\mc T'$ have cardinality at least 2. On the other hand, by condition \ref{cond1}, we have $t_a t_b \in \ini_{\prec_h} I_{\mc A}$ for all $\{a,b\} \in \mc E \setminus \mc E'$. Thus $\ini_{\prec_h} I_{\mc A} = \langle t_a t_b : \{a,b\} \in \mc E \setminus \mc E' \rangle$. It follows that $h$ induces a quadratic triangulation of $\mc A$, as desired.
\end{proof}

\begin{rem}
Theorem~\ref{thm:main conditions} can be extended to general finite subsets of $\bb Z^d$ by adding conditions that enforce that every lattice point in $\conv(\mc A)$ appears in the triangulation. (For example, by enforcing the appropriate inequality on circuits where one part has size one.)
\end{rem}

\section{Computations}
\subsection{No quadratic triangulation for the Fano matroid}\label{subsec:noQTFano}

Our main method is to turn the conditions from Theorem \ref{thm:main conditions} into formal statements and to prove that these statements cannot be satisfied. This is achieved in the following steps.

\paragraph{Equivalence classes}
As the matroid polytope $P(F_7)$ of the Fano matroid has $28$ vertices, the set $\mc E$ has cardinality $\binom{28}{2} = 378$. 

\begin{prop}
The set $\mc E$ is subdivided into the following equivalence classes:
\begin{enumerate}
\item $126$ classes of cardinality $1$, corresponding to edges of $P(F_7)$ or to symmetric basis exchanges---pairs of bases of the matroid that differ by one element; 
\item $105$ classes of cardinality $2$ corresponding to pairs of bases that share one element. There are $7$ possibilities to choose the shared element and $\binom{6}{4}=15$ possibilities to choose the remaining $4$ elements for the two bases;
\item $7$ classes of cardinality $6$ corresponding to pairs of disjoint bases.
\end{enumerate}
\end{prop}
\begin{proof}
    The statements are easy to check directly and we only sketch the proof. The number of edges of $P(F_7)$ is known. As the midpoint of each edge is not a midpoint of another pair of vertices, they have to form singleton classes. For the second point, choose $x\in [7]$ and a four element subset $A\subset [7]\setminus\{x\}$. There are three partitions of $A$ as a disjoint union of two element sets. However, only two of those partitions give bases of $F_7$ when considered together with $x$. These two clearly belong to the same equivalence class.

    Finally, for the last point, we have seven possibilities of choosing a subset $A\subset [7]$ of cardinality $6$. There are $6$ ways of presenting $A$ as a disjoint union of two bases. 
\end{proof}

\paragraph{Circuits}\label{subsec:circuits}
We compute all signed circuits of (the vertices of) $P(F_7)$. There are $181734$ circuits. The results are obtained independently using TOPCOM \cite{TOPCOM-preview} and Macaulay2 \cite{Macaulay2, FourTiTwoM2}. The file containing all circuits is available on \url{https://github.com/gaku-liu/quadratic_triangulations/blob/main/circuitsFano.txt}. We note that small circuits $(S_+,S_-)$ with $|S_+|=|S_-|=2$ correspond to elements equivalent in $\mc E$, as they encode a relation among the vertices of type $a+b=c+d$. 

\paragraph{Formal statements}
Our next aim is to express the conditions we have as formal statements. For this we declare two types of variables:
\begin{enumerate}
    \item $378$ Boolean variables $x_{i,j}$ indexed by the elements of the set $\mc E$, i.e.~pairs of vertices of $P(F_7)$ and
    \item $28$ Real variables $h_j$ indexed by the vertices of $P(F_7)$.
\end{enumerate}

The boolean variables $x_{i,j}$ represent the transversal, i.e.~$x_{i,j}$ is True if $\{i,j\}$ is in the transversal and False otherwise. The variables $h_j$ represent the function $h$. Statement $(1)$ from Theorem \ref{thm:main conditions} becomes:
\begin{itemize}
    \item For each equivalence class pick two distinct elements $\{a,b\},\{a',b'\}$ in that class. We assert: If $x_{a',b'}$, then $h_{a'}+h_{b'}<h_a+h_b$. 
\end{itemize}
Statement $(2)$ becomes:
\begin{itemize}
    \item For every circuit $(C_+,C_-)$ let $S=\binom{C_+}{2}\cup\binom{C_-}{2}$ be the set of pairs $\{i,j\}$ such that $i,j\in C_+$ or $i,j\in C_-$. We assert $\bigvee_{\{i,j\}\in S} \neg x_{i,j}$.
\end{itemize}

We further would like to obtain a transversal, that is we want  to choose exactly one representative from each equivalence class. The conditions (1) already imply that each equivalence class has at most one representative, so we need only to enforce that each class has at least one:
\begin{itemize}
    \item For every equivalence class $S$ on $\mc E$, we assert $\bigvee_{\{i,j\}\in S} x_{i,j}$.
\end{itemize}

\begin{rem}
While our Theorem \ref{thm:main conditions} gives an if and only if condition, we observed that adding statements that are consequences of having a quadratic triangulation may make the solver run faster. Here are examples of such consequences:
\begin{itemize}
    \item For every equivalence class $S$ on $\mc E$, we assert $\bigvee_{e_1\neq e_2\in S}\neg (x_{e_1}\land x_{e_2})$.
    \item For any minimal nonunimodular simplex, let $S$ be the set of its edges. As a quadratic triangulation cannot contain a nonunimodular simplex, we assert $\bigvee_{\{i,j\}\in S}\neg x_{i,j}$.
\end{itemize}
\end{rem}

\paragraph{Satisfiability Modulo Theories: cvc5}
Finally, we may put the formal statements into an SMT solver, for which we use cvc5  \cite{cvc5}.

We used Python to produce a script in the SMT-LIB language, which was then fed into cvc5. The Python code is available at \url{https://github.com/gaku-liu/quadratic_triangulations/blob/main/to_smt_general.py} and \url{https://github.com/gaku-liu/quadratic_triangulations/blob/main/triangulation_tools.py}.

The proof of unsatisfiability, without parallelization, takes about fifteen minutes on a 2024 M4 Mac mini with 16GB RAM. The resulting output file containing the proof has over 1 GB. This finishes the proof of Theorem \ref{thm: main fano}.


\section{Towards a characterization of matroids with quadratic triangulations: 8 element matroids}\label{subsec: 8 elem}

Let $\mathcal{Q}$ be the class of all matroids whose matroid base polytope admits a quadratic triangulation.  It is natural to ask whether there exists a characterization of $\mathcal{Q}$.  We begin with a simple lemma.

\begin{lem}
The class $\mathcal{Q}$ is closed under taking minors and duals.
\end{lem}

\begin{proof}
Given a matroid $M$, the base polytopes of minors of $M$ are distinguished faces of $P(M)$.  Additionally, $P(M^*)=-P(M)+{\bf 1}$.  Thus it is clear that if $M$ belongs to $\mathcal{Q}$, then so do all of its minors as well as its dual.  
\end{proof}

We have the following natural question.

\begin{ques}
Is the class $\mathcal{Q}$ determined by a finite, self-dual list of forbidden minors?
\end{ques}

Knowing that the Fano matroid and its dual do not admit quadratic Gr\"obner bases, but all other matroids on at most 7 elements do, we have the following strengthened version of the above question.

\begin{ques}\label{ques:excludedfanomino}
Is $\mathcal{Q}$ equal to the class of matroids which are $F_7$ and $F_7^*$ minor free?
\end{ques}

In hopes of gaining some insight into these questions, we set out to investigate matroids on 8 elements.  It is notable that there is a significant jump in complexity when passing from matroids on 7 elements to matroids on 8 elements; matroid base polytopes for 7 element matroids may have a few hundred thousand circuits while  matroid base polytopes for 8 element matroids may have a few billion.   Here we report on the progress we have made which can be summarized in the following result.

\begin{thm}\label{thm:8element}
Let $M$ be a matroid on at most 8 elements which has no $F_7$ or $F^*_7$ minor, and is not equal to the matroid $T_8$, then $P(M)$ admits a quadratic triangulation.
\end{thm}

The reason for excluding $T_8$ in Theorem \ref{thm:8element} is that we have been unable to ascertain whether it has a quadratic triangulation.  So, Theorem \ref{thm:8element} does not provide a negative answer to Question \ref{ques:excludedfanomino}.   Below we outline our computations which were performed on the Vermont Advanced Computing Cluster (VACC).  

\paragraph{The matroids considered} By Shibata \cite{Shibata2016}, to understand whether a given matroid $M$ belongs to $\mathcal{Q}$, it suffices to decompose $M$ into 3-connected components and check these 3-connected components individually.  
By Blum \cite{blum2001base}, we know that all base sortable matroids admit quadratic triangulations.  In particular, this implies that we are not interested in ranks 0, 1, 2, 6, 7, or 8 as these are all base sortable.  
Thus, we first generated a Sage script to calculate the list of all 3-connected matroids on 8 elements which have no $F_7$ or $F^*_7$ minor, are not base sortable, and have rank 3 or 4 (with rank 5 being subsumed by rank 3 via duality).  There are 478 such matroids; 43 of rank 3 and 435 of rank 4.\footnote{This calculation was independently verified by a separate script.}    Let this collection of $478$ matroids be denoted by $\mathcal{R}$.

 

\paragraph{Circuit preprocessing} We first recalculated the full lists of circuits for the hypersimplices $\Delta(8,3)$ and $\Delta(8,4)$ up to symmetry using TOPCOM 1.2.0.c.; these were previously calculated by Rambau \cite{RambauSymLexSubsetRS2026}.  Afterwards, we found the full list of circuits by taking the orbits of the representatives TOPCOM provided. 
The hypersimplex $\Delta(8,3)$ has 251,651,820 circuits which were stored in a 7GB file, and the hypersimplex $\Delta(8,4)$ has 3,134,451,775 circuits which were stored in a 90GB file.  

 \paragraph{The SMT program}For a given matroid $M$, the pipeline $\mathcal{X}$ computes a subset of the circuits of $P(M)$ by restricting to circuits of sizes 4, 5, and 6 from the relevant hypersimplex $\Delta(8,3)$ or $\Delta(8,4)$.  Let the circuits produced be denoted $\mathcal{S}(M)$.\footnote{This totals only about 0.1\% of all circuits for a given matroid!} 
 Then $\mathcal{X}$ streams these circuits $\mathcal{S}(M)$ directly to cvc5 which determines whether the full minimum-edge conditions \ref{cond1} from Theorem \ref{thm:main conditions} together with the circuit conditions \ref{cond2} restricted to $\mathcal{S}(M)$ are satisfiable.  If cvc5 produces UNSAT it terminates and we conclude that $P(M)$ has no quadratic triangulation (this never happened in our cases).  If cvc5 finds a SAT witness $(h,E)$, that gets passed to the quadratic triangulation verification (QTV) step.

\paragraph{Quadratic triangulation verification step}  
Let $(h,E)$ be the SAT witness pair produced by cvc5 for the relaxed program. The QTV stage first perturbs $h$, call this $h^*$, and the checks whether $h^*$ produces a quadratic triangulation with edge set $E$.  In the process it checks whether the facets of the clique complex of $E$ belong to the lower hull of $h^*$, but it does not directly calculate the full lower convex hull of $h^*$.  If $h^*$ does not produce a quadratic triangulation with edge set $E$, we go back to the circuit generation step, increase the partial set of circuits by some predetermined increment, and repeat.  We note that all but one of the matroids passed the cvc5 and QTV steps after a single iteration. This was quite surprising and we have no theoretical justification for this phenomenon.

\paragraph{Full circuit audit} If the QTV step passes, we are done and may conclude that $P(M)$ has a quadratic triangulation induced by the height function $h$ with edge set $E$.  However, we added a second verification stage, called the full-circuit audit, to ensure that the witnesses produced are valid.  This takes the output edge set and directly verifies that it satisfies the full set of circuit constraints for $P(M)$ in condition \ref{cond2} from Theorem \ref{thm:main conditions}.  We did this as a second campaign which ran as a parallel computation: 8 CPUs for each rank 3 matroid, and 16 CPUs for each rank 4 matroid.  \footnote{In fact, a third check was performed by having Macaualay2 calculate the toric ideals $I_M$ for $M \in \mathcal{R}$ and verify directly that the height functions $h^*$ above give a term order such that the associated Gr\"obner basis is quadratic.}

\paragraph{Main results}
In total, 477 of the 478 matroids in $\mathcal{R}$ successfully terminated with a SAT witness which passed the QTV step and the full circuit audit; the one matroid which did not terminate was $T_8$.  For the main SAT + QTV pipeline, the median rank 3 matroid took less than 1 minute and the median rank 4 matroid took less than 9 minutes.  The rank 3 and 4 matroids used about 150 MB and 500 MB of RAM, respectively.

There were a very small number of matroids which were considerably harder than the others to verify.  The rank 3 matroid $AG(2,3)\setminus e$ took about 2.5 hours to finish.   There were two matroids of rank 4, the $Y$-$\Delta$ transform of $AG(2,3)\setminus e$, and an unnamed matroid which is one basis away from $T_8$, which took 20 and 30 hours to finish, respectively.  We have not been able to determine whether the matroid polytope of $T_8$ has a quadratic triangulation.  We ran several week-long jobs to test this case and they all timed out.  See \url{https://github.com/gaku-liu/quadratic_triangulations/blob/main/witnesses_8elt_matroids.zip} for further details.

\section{Open problems}

We recall that the standard graded algebra $S=R/I$, where $R=k[x_1,\dots,x_n]$ is Koszul if and only if $k=R/(x_1,\dots,x_n)$ has a linear resolution as an $S$-module.

Let $I_M\subset R$ be the toric ideal of a matroid $M$. The following open problem was posed by Herzog and Hibi \cite{herzog2002discrete}.
\begin{open}\label{open:Koszul}
    Is every base ring $R/I_M$ Koszul?
\end{open}
It is well-known that a positive answer to this open problem implies that $I_M$ is quadratically generated. On the other hand, a quadratic Gr\"obner basis for $I_M$ implies that $R/I_M$ is Koszul. It is thus natural to ask:
\begin{open}
    Is the base ring of the Fano matroid Koszul?
\end{open}
Given the toric ideal $I$ of the Fano matroid we computed first part of the resolution of $k$ as an $S$-algebra. This part remained linear:

\[
\cdots \longrightarrow S(-4)^{133343} \longrightarrow S(-3)^{8463} \longrightarrow S(-2)^{518} \longrightarrow S(-1)^{28} \longrightarrow S \longrightarrow k \longrightarrow 0
\]
These computations suggest a positive answer to the above question.

We propose the following problems based on Subsection \ref{subsec: 8 elem}.  

\begin{open} \label{T8}
Determine whether the matroid $T_8$ admits a quadratic triangulation.
\end{open}

\begin{open}
Determine whether the matroid $AG(2,3)$ admits a quadratic triangulation.
\end{open}

Our data suggests a rather vague but noticeable connection to representability theory: $F_7$, $AG(2,3)\setminus e$, the $Y$-$\Delta$ transform of $AG(2,3)\setminus e$, $T_8$ and their duals belong to the forbidden minors for the so-called dyadic matroids.  These are the matroids which are representable over the dyadic partial field, equivalently they are representable over all fields of odd characteristic (see \cite{BakerBowler2019, BrettellPendavingh2024, SempleWhittle1996}).  
In comparison, $F_7$ does not have a quadratic triangulation. The matroids $AG(2,3)\setminus e$ and the $Y$-$\Delta$ transform of $AG(2,3)\setminus e$ do have quadratic triangulations, but they were hard to find and so one might expect that there are very few.

\begin{conj}\label{conj:dyadic}
The matroid base polytope of any dyadic matroid has a quadratic triangulation.   
\end{conj}

Note that there are matroids which are not dyadic, but admit quadratic triangulations, so Conjecture \ref{conj:dyadic} does not suggest a characterization of $\mathcal{Q}$.

Theorem $3.2$ of \cite{Lason2021FixedRankMatroidToricIdeals} asserts that the toric ideal of a matroid of rank $r$ has a Gröbner basis of degree at most $(r+3)!$. Since the Fano matroid witnesses that degree 2 is not a universal bound, we ask the following.

\begin{open}
Does the toric ideal of every matroid admit a Gr\"{o}bner basis whose degree is bounded by a universal constant, independent of the matroid?
\end{open}

\bibliographystyle{plain}
\bibliography{references}

\end{document}